\documentclass[11pt]{amsart}

\usepackage[utf8]{inputenc}
\usepackage[T1]{fontenc}
\usepackage[english]{babel}
\usepackage{amsmath,amssymb,amsthm}
\usepackage{mathtools}
\usepackage[shortlabels]{enumitem}
\usepackage[hidelinks]{hyperref}

\newtheorem{theorem}{Theorem}[section]
\newtheorem{proposition}[theorem]{Proposition}
\newtheorem{lemma}[theorem]{Lemma}
\newtheorem{corollary}[theorem]{Corollary}
\newtheorem{maintheorem}{Theorem}

\theoremstyle{definition}

\newtheorem{example}[theorem]{Example}

\theoremstyle{remark}
\newtheorem{remark}[theorem]{Remark}

\numberwithin{equation}{section}

\newcommand{\R}{\mathbb{R}}
\newcommand{\C}{\mathbb{C}}
\newcommand{\K}{\mathbb{K}}
\newcommand{\re}{\operatorname{Re}}
\newcommand{\ip}[2]{\langle #1, #2 \rangle}
\newcommand{\conj}[1]{\overline{#1}}
\newcommand{\dbar}{\bar\partial}
\newcommand{\din}{d_{\mathrm{in}}}
\newcommand{\dout}{d_{\mathrm{out}}}
\DeclareMathOperator{\Sym}{Sym}

\begin{document}

\title[Inner and outer isometries of graphs of polynomials]{Inner and outer isometries\\ of graphs of polynomials}

\author{Alexandre Fernandes}
\address{Departamento de Matem\'atica, Universidade Federal do Cear\'a, Fortaleza, CE, Brazil}
\email{alex@mat.ufc.br}

\author{Claudia Fernandes}
\address{Centro de Ci\^encias e Tecnologia, Universidade Estadual do Cear\'a, Fortaleza, CE, Brazil}
\email{claudia.reboucas@uece.br}

\author{Edson Sampaio}
\address{Departamento de Matem\'atica, Universidade Federal do Cear\'a, Fortaleza, CE, Brazil}
\email{edsonsampaio@mat.ufc.br}


\thanks{The first named author was partially supported by CNPq-Brazil grant  304700/2021-5 and 405618/2023-9. The second named author was partially supported by CNPq-Brazil grants 303375/2025-6, 404611/2025-7 and 405618/2023-9 and by the Serrapilheira Institute (grant number Serra -- R-2110-39576).
}

\date{\today}

\begin{abstract}
We classify graphs of complex polynomials in one variable up to
isometry, both for the inner (length) distance and for the outer
(Euclidean) distance of $\C^2\cong\R^4$. For polynomials $p,q$ with
$\deg p\ge 2$ we show that the graphs of $p$ and $q$ are inner
isometric if and only if they are outer isometric, if and only if
$p(z)=\alpha\,q(\beta z+h)+k$ or $p(z)=\alpha\,\bar q(\beta z+h)+k$ for
some $\alpha,\beta,h,k\in\C$ with $|\alpha|=|\beta|=1$, where $\bar q$
denotes the polynomial with conjugate coefficients. Moreover, every inner
isometry between such graphs is the restriction of an isometry of $\C^2$;
in this sense graphs of complex polynomials are isometrically rigid. We
contrast this with the real case, where the two distances behave very
differently: the graphs of any two real polynomials are inner isometric,
whereas, for degree at least $2$, outer isometries are generated by
translations, reflections in horizontal and vertical lines and
half-turns.
\end{abstract}

\maketitle

\section{Introduction}

Let $\K=\R$ or $\C$. The graph of a polynomial $p\in\K[t]$ is the set
\[
  G_p=\{(t,p(t)) : t\in\K\}\subset\K^2,
\]
which we regard as a subset of the Euclidean space $\R^2$ or
$\C^2\cong\R^4$. Following the usual terminology of Lipschitz geometry
(see e.g.~\cite{BM}), a subset $X$ of a Euclidean space carries two
natural metrics: the \emph{outer distance} $\dout$, which is the
restriction of the Euclidean distance, and the \emph{inner distance}
$\din$, where $\din(u,v)$ is the infimum of the lengths of curves in $X$
joining $u$ to $v$. We say that $G_p$ and $G_q$ are \emph{inner} (resp.\
\emph{outer}) \emph{isometric} if there is a bijection $G_p\to G_q$
preserving the inner (resp.\ outer) distances.

The question that motivated this paper was asked by Edson Sampaio during
the PhD qualifying examination of his student Henrique Blanco: \emph{when
are the graphs of two polynomials isometric?} Some answers are
immediate. Isometries of the ambient space such as translations, the
maps $(z,w)\mapsto(\beta z,\bar\alpha w)$ with $|\alpha|=|\beta|=1$, and
complex conjugation $(z,w)\mapsto(\bar z,\bar w)$ carry graphs of
polynomials onto graphs of polynomials. Our main result says that, in the
complex case and for degree at least $2$, these are the only
possibilities, for either of the two distances.

\begin{maintheorem}\label{thm:A}
Let $p,q\in\C[z]$ with $\deg p\ge 2$. The following conditions are
equivalent:
\begin{enumerate}[(a)]
\item $G_p$ and $G_q$ are inner isometric;
\item $G_p$ and $G_q$ are outer isometric;
\item there exist $\alpha,\beta,h,k\in\C$ with $|\alpha|=|\beta|=1$ such
  that
  \[
    \begin{aligned}
      &p(z)=\alpha\,q(\beta z+h)+k &&\text{for all } z\in\C,\ \text{or}\\
      &p(z)=\alpha\,\bar q(\beta z+h)+k &&\text{for all } z\in\C.
    \end{aligned}
  \]
\end{enumerate}
Moreover, every inner isometry $\Gamma\colon G_p\to G_q$ is the
restriction of an isometry of $\C^2$; precisely, $\Gamma$ is the
restriction of
\[
  (z,w)\mapsto\bigl(\beta z+h,\ \bar\alpha(w-k)\bigr)
  \qquad\text{or}\qquad
  (z,w)\mapsto\bigl(\conj{\beta z+h},\ \conj{\bar\alpha(w-k)}\bigr),
\]
with $\alpha,\beta,h,k$ as in \textup{(c)}.
\end{maintheorem}

Thus graphs of complex polynomials of degree at least $2$ are
\emph{isometrically rigid}: they admit no inner isometries other than
those induced by the ambient space, and the inner and outer
classifications coincide. The graph of a polynomial of degree at most $1$
is a real affine $2$-plane, and any two such planes are isometric for
both distances (Remark~\ref{rem:degree-one}).

The proof of the implication (a)$\Rightarrow$(c) has two steps. First,
the inner metric of $G_p$ is the Riemannian metric
$(1+|p'(z)|^2)\,|dz|^2$ on $\C$, which is conformal; by the
Myers--Steenrod theorem and an elementary argument with Wirtinger
derivatives, an isometry between two conformal metrics on the whole
plane $\C$ is affine or anti-affine (Proposition~\ref{prop:conformal}).
Second, the resulting identity
$|a|^2(1+|q'(az+b)|^2)=1+|p'(z)|^2$ is analysed by comparing
coefficients in $\C[z,\bar z]$ (Lemma~\ref{lem:moduli}). The implication
(b)$\Rightarrow$(a) holds because outer isometries extend to isometries
of the ambient space (Lemma~\ref{lem:extension}).

In the real case the two distances behave very differently. A curve has
no intrinsic geometry, and the graph of any real polynomial,
parametrized by arc length, is inner isometric to the real line. On the
other hand, the outer classification is again rigid.

\begin{maintheorem}\label{thm:B}
Let $p,q\in\R[x]$.
\begin{enumerate}[(a)]
\item The graphs $G_p$ and $G_q$ are always inner isometric.
\item If $\deg p\ge 2$, then $G_p$ and $G_q$ are outer isometric if and
  only if there exist $a,b\in\{1,-1\}$ and $h,k\in\R$ such that
  $p(x)=a\,q(bx+h)+k$ for all $x\in\R$.
\end{enumerate}
\end{maintheorem}

Part (b) says that outer isometries between graphs of real polynomials of
degree at least $2$ are generated by translations, reflections in
horizontal and vertical lines and half-turns; no rotation by an angle
other than $0$ or $\pi$ occurs. As a consequence, the symmetry group of
such a graph has at most two elements (Corollary~\ref{cor:symmetry}).
Part (a) is elementary, but it contradicts a common intuition: the
curvature of a plane curve is invariant under rigid motions, but it is
not an inner invariant; for instance, the parabola $y=x^2$ is inner
isometric to a line. It is precisely because graphs of complex
polynomials are surfaces, whose Gaussian curvature is intrinsic, that
the complex case behaves differently.

Comparing Theorems~\ref{thm:A} and~\ref{thm:B}, note that for $p,q$ with
real coefficients one has $\bar q=q$, so the complex graphs are
isometric if and only if $p(z)=\alpha\,q(\beta z+h)+k$ with
$|\alpha|=|\beta|=1$. This is strictly weaker than Theorem~\ref{thm:B}(b):
the complex graphs of $x^3+x$ and $x^3-x$ are isometric, but their real
graphs are not outer isometric (Example~\ref{ex:cubic}).

To the best of our knowledge, these statements are natural, but we were not able to find them explicitly in the literature..

The paper is organized as follows. Section~\ref{sec:prelim} collects
preliminary results: extension of outer isometries, polynomials in $z$
and $\bar z$, and isometries between conformal metrics on $\C$.
Section~\ref{sec:complex} contains the proof of
Theorem~\ref{thm:A}, and Section~\ref{sec:real} that of
Theorem~\ref{thm:B}. Section~\ref{sec:final} contains some remarks and
questions.

\section{Preliminaries}\label{sec:prelim}

Throughout, $\ip{\cdot}{\cdot}$ and $|\cdot|$ denote the Euclidean inner
product and norm of $\R^N$, and $O(N)$ the orthogonal group. We identify
$\C^2$ with $\R^4$; the corresponding inner product is
$\ip{(z_1,w_1)}{(z_2,w_2)}=\re(z_1\bar z_2+w_1\bar w_2)$. For a
polynomial $q=\sum q_jz^j\in\C[z]$ we write $\bar q=\sum\bar q_jz^j$, so
that $\conj{q(z)}=\bar q(\bar z)$.

\subsection{Outer and inner isometries}

\begin{lemma}\label{lem:extension}
Let $X\subset\R^N$ and let $f\colon X\to\R^N$ preserve Euclidean
distances. Then there are $A\in O(N)$ and $c\in\R^N$ with $f(v)=Av+c$ for
all $v\in X$.
\end{lemma}

\begin{proof}
Fix $v_0\in X$ and set $Y=X-v_0$ and $g(v)=f(v+v_0)-f(v_0)$ for $v\in Y$.
Then $g(0)=0$ and $g$ preserves distances, so $|g(v)|=|v|$, and expanding
$|g(u)-g(v)|^2=|u-v|^2$ gives $\ip{g(u)}{g(v)}=\ip{u}{v}$ for
$u,v\in Y$. Let $V$ be the linear span of $Y$, choose a basis
$u_1,\dots,u_k\in Y$ of $V$, and let $A_0\colon V\to\R^N$ be the linear
map with $A_0u_i=g(u_i)$. Then $A_0$ preserves inner products. For
$v\in Y$ we have
$\ip{g(v)}{A_0u_i}=\ip{v}{u_i}=\ip{A_0v}{A_0u_i}$ for all $i$, so
$g(v)-A_0v\perp A_0(V)$, whence
$|g(v)|^2=|A_0v|^2+|g(v)-A_0v|^2$; since $|g(v)|=|v|=|A_0v|$, we get
$g(v)=A_0v$. Finally, $A_0$ extends to some $A\in O(N)$ by mapping an
orthonormal basis of $V^\perp$ onto one of $A_0(V)^\perp$, and
$f(v)=Av+c$ with $c=f(v_0)-Av_0$.
\end{proof}

For $p\in\K[t]$ let $F_p\colon\K\to G_p$, $F_p(t)=(t,p(t))$, and let
$\pi(u,w)=u$ be the first projection, so that $F_p$ is a smooth embedding
with inverse $\pi|_{G_p}$. A \emph{curve} in $G_p$ is a continuous map
$\gamma\colon[0,1]\to G_p$ such that $\pi\circ\gamma$ is piecewise $C^1$;
its length is $L(\gamma)=\int_0^1|\gamma'(t)|\,dt$, and $\din$ is defined
using such curves. Every curve in $G_p$ has the form $\gamma=F_p\circ c$
with $c=\pi\circ\gamma$, and $\gamma'=(c',\,p'(c)\,c')$, so that
\begin{equation}\label{eq:length}
  L(F_p\circ c)=\int_0^1\sqrt{1+|p'(c(t))|^2}\;|c'(t)|\,dt .
\end{equation}
In other words, $F_p$ is an isometry from $\K$, endowed with the length
distance of the Riemannian metric
\begin{equation}\label{eq:metric}
  g_p=\bigl(1+|p'(t)|^2\bigr)\,|dt|^2,
\end{equation}
onto $(G_p,\din)$.

\begin{lemma}\label{lem:outer-inner}
Every outer isometry $\Gamma\colon G_p\to G_q$ is an inner isometry.
\end{lemma}

\begin{proof}
By Lemma~\ref{lem:extension}, $\Gamma$ is the restriction of an affine
isometry $E$ of the ambient space. If $\gamma$ is a curve in $G_p$, then
$E\circ\gamma$ is a curve in $G_q$ of the same length, and the same holds
for $E^{-1}$. Hence $\Gamma$ preserves $\din$.
\end{proof}

\subsection{Polynomials in \texorpdfstring{$z$}{z} and \texorpdfstring{$\bar z$}{z-bar}}

A polynomial expression $P=\sum c_{jk}z^j\bar z^k$ with $c_{jk}\in\C$
defines a function $\C\to\C$. Since $z$ and $\bar z$ are not independent
functions, the following fact, used systematically below, deserves a
proof.

\begin{lemma}\label{lem:zzbar}
If $P=\sum c_{jk}z^j\bar z^k$ vanishes at every $z\in\C$, then
$c_{jk}=0$ for all $j,k$. Consequently, identities between polynomials in
$z$ and $\bar z$ that hold for all $z\in\C$ may be checked coefficientwise,
and we may regard $\C[z,\bar z]$ as a polynomial ring in two variables.
\end{lemma}

\begin{proof}
The ring homomorphism $\C[z,w]\to\C[x,y]$, $z\mapsto x+iy$,
$w\mapsto x-iy$, is an isomorphism, with inverse $x\mapsto(z+w)/2$,
$y\mapsto(z-w)/(2i)$. The polynomial $Q(x,y)=P(x+iy,x-iy)$ vanishes on
$\R^2$; writing $Q=\sum_k a_k(x)\,y^k$, for each real $x$ the polynomial
$y\mapsto Q(x,y)$ has infinitely many roots, so $a_k(x)=0$, and then each
$a_k$ has infinitely many roots, so $a_k=0$. Hence $Q=0$ and $P=0$.
\end{proof}

Alternatively, in terms of the Wirtinger derivatives introduced below,
$\partial^j\dbar^kP(0)=j!\,k!\,c_{jk}$.

\begin{lemma}\label{lem:moduli}
Let $f,g\in\C[z]$ with $f$ non-constant, and suppose that
$|f(z)|^2-|g(z)|^2=c$ for all $z\in\C$, where $c$ is a constant. Then
$c=0$ and $f=v\,g$ for some $v\in\C$ with $|v|=1$.
\end{lemma}

\begin{proof}
Write $f=\sum f_jz^j$ and $g=\sum g_jz^j$. By Lemma~\ref{lem:zzbar}, the
identity $\sum_{j,k}(f_j\bar f_k-g_j\bar g_k)\,z^j\bar z^k=c$ gives
$f_j\bar f_k=g_j\bar g_k$ for all $(j,k)\ne(0,0)$. Choose $j\ge 1$ with
$f_j\ne 0$. Taking $k=j$ gives $|g_j|=|f_j|\ne 0$, and then, for every
$k\ge 0$, $f_k=v\,g_k$ with $v=\conj{g_j/f_j}$, which has modulus $1$.
Therefore $|f|^2=|g|^2$ and $c=0$.
\end{proof}

\subsection{Isometries between conformal metrics on \texorpdfstring{$\C$}{C}}

For $\varphi\colon\C\to\C$ of class $C^1$, the Wirtinger derivatives are
\[
  \partial\varphi=\tfrac12(\varphi_x-i\varphi_y),\qquad
  \dbar\varphi=\tfrac12(\varphi_x+i\varphi_y).
\]
Writing $u=u_1+iu_2$ with
$u_1=(u+\bar u)/2$ and $u_2=(u-\bar u)/(2i)$ in
$d\varphi_z(u)=\varphi_xu_1+\varphi_yu_2$, one gets
\begin{equation}\label{eq:wirtinger}
  d\varphi_z(u)=\partial\varphi(z)\,u+\dbar\varphi(z)\,\bar u
  \qquad(z,u\in\C),
\end{equation}
and $\dbar\varphi\equiv 0$ is the Cauchy--Riemann system.

\begin{lemma}\label{lem:linear}
Let $L(u)=Au+B\bar u$ with $A,B\in\C$, an $\R$-linear map of
$\C\cong\R^2$.
\begin{enumerate}[(i)]
\item The coefficients are determined by $L$, and
  $\det L=|A|^2-|B|^2$.
\item If $|L(u)|=\mu|u|$ for all $u\in\C$ and some $\mu>0$, then
  $A\bar B=0$.
\end{enumerate}
\end{lemma}

\begin{proof}
(i) $L(1)=A+B$ and $L(i)=i(A-B)$ determine $A$ and $B$. With
$A=a_1+ia_2$ and $B=b_1+ib_2$, the matrix of $L$ in the basis $\{1,i\}$
has columns $(a_1+b_1,\,a_2+b_2)$ and $(b_2-a_2,\,a_1-b_1)$, whose
determinant is $a_1^2+a_2^2-b_1^2-b_2^2$.
(ii) We have $|Au+B\bar u|^2=(|A|^2+|B|^2)|u|^2+2\re(A\bar B\,u^2)$.
Taking $u=e^{i\theta}$, the function
$\theta\mapsto\re(A\bar B\,e^{2i\theta})$ must be constant, which forces
$A\bar B=0$.
\end{proof}

\begin{lemma}\label{lem:entire}
Every injective entire function $\varphi\colon\C\to\C$ has the form
$\varphi(z)=az+b$ with $a\ne 0$.
\end{lemma}

\begin{proof}
If $\varphi$ were not a polynomial, it would have an essential
singularity at infinity, and by the Casorati--Weierstrass theorem
$\varphi(\{|z|>1\})$ would be dense in $\C$; but it is disjoint, by
injectivity, from the non-empty open set $\varphi(\{|z|<1\})$. Thus
$\varphi$ is a polynomial, and an injective polynomial has degree $1$ by
the fundamental theorem of algebra (see e.g.~\cite{Ahlfors}).
\end{proof}

\begin{proposition}\label{prop:conformal}
Let $\lambda,\nu\colon\C\to(0,\infty)$ be continuous, and let
$\varphi\colon\C\to\C$ be a $C^1$ diffeomorphism such that
\begin{equation}\label{eq:pullback}
  \nu(\varphi(z))\,|d\varphi_z(u)|^2=\lambda(z)\,|u|^2
  \qquad\text{for all } z,u\in\C,
\end{equation}
that is, $\varphi^*(\nu\,|dw|^2)=\lambda\,|dz|^2$. Then there are
$a,b\in\C$, $a\ne 0$, such that either $\varphi(z)=az+b$ for all $z$ or
$\varphi(z)=a\bar z+b$ for all $z$. In both cases
$\nu(\varphi(z))\,|a|^2=\lambda(z)$.
\end{proposition}

\begin{proof}
By~\eqref{eq:pullback}, for each $z$ the $\R$-linear map $d\varphi_z$
multiplies norms by $\mu=\sqrt{\lambda(z)/\nu(\varphi(z))}>0$. By
\eqref{eq:wirtinger} and Lemma~\ref{lem:linear}(ii),
$\partial\varphi(z)\,\conj{\dbar\varphi(z)}=0$ for every $z$. By
Lemma~\ref{lem:linear}(i),
$\det d\varphi_z=|\partial\varphi(z)|^2-|\dbar\varphi(z)|^2$; this is a
continuous function which never vanishes, because $\varphi$ is a
diffeomorphism, so it has constant sign on the connected set $\C$. If it
is positive, then $\dbar\varphi\equiv 0$, so $\varphi$ is holomorphic,
and Lemma~\ref{lem:entire} gives $\varphi(z)=az+b$. If it is negative,
then $\partial\varphi\equiv 0$; the map $\psi(z)=\varphi(\bar z)$ satisfies
$d\psi_z(u)=d\varphi_{\bar z}(\bar u)
=\partial\varphi(\bar z)\,\bar u+\dbar\varphi(\bar z)\,u$, so by
Lemma~\ref{lem:linear}(i) $\dbar\psi(z)=\partial\varphi(\bar z)=0$, and
$\psi$ is an injective entire function. Hence $\psi(z)=az+b$ and
$\varphi(z)=a\bar z+b$. In both cases $|d\varphi_z(u)|=|a|\,|u|$, and
\eqref{eq:pullback} gives $\nu(\varphi(z))\,|a|^2=\lambda(z)$.
\end{proof}

\begin{remark}
Proposition~\ref{prop:conformal} uses in an essential way that the
domain is the whole plane: on a proper domain, such as the unit disc with
its hyperbolic metric, there are conformal isometries which are not
affine.
\end{remark}

\section{The complex case}\label{sec:complex}

For $p\in\C[z]$, the metric~\eqref{eq:metric} is
$g_p=(1+|p'(z)|^2)\,|dz|^2$, conformal to the Euclidean metric of $\C$.
Its Gaussian curvature is
\begin{equation}\label{eq:curvature}
  K_p=-\frac{\Delta\log(1+|p'|^2)}{2\,(1+|p'|^2)}
  =-\frac{2\,|p''|^2}{\bigl(1+|p'|^2\bigr)^3},
\end{equation}
where $\Delta=4\partial\dbar$; see~\cite{doCarmo} for the curvature of a
conformal metric, and note that
$\partial\dbar\log(1+|f|^2)=|f'|^2/(1+|f|^2)^2$ for holomorphic $f$. In
particular $K_p\le 0$, and $G_p$ is flat if and only if $\deg p\le 1$.
Formula~\eqref{eq:curvature} is not used in the proofs, but it shows that,
unlike the real case, graphs of complex polynomials have non-trivial
inner geometry.

\begin{theorem}\label{thm:inner-rigidity}
Let $p,q\in\C[z]$ with $\deg p\ge 2$, and let $\Gamma\colon G_p\to G_q$
be an inner isometry. Then there exist $\alpha,\beta,h,k\in\C$ with
$|\alpha|=|\beta|=1$ such that one of the following holds:
\begin{enumerate}[(i)]
\item $p(z)=\alpha\,q(\beta z+h)+k$ for all $z$, and
  $\Gamma(z,w)=\bigl(\beta z+h,\ \bar\alpha(w-k)\bigr)$ on $G_p$;
\item $p(z)=\alpha\,\bar q(\beta z+h)+k$ for all $z$, and
  $\Gamma(z,w)=\bigl(\conj{\beta z+h},\ \conj{\bar\alpha(w-k)}\bigr)$ on
  $G_p$.
\end{enumerate}
\end{theorem}

\begin{proof}
Let $\varphi=F_q^{-1}\circ\Gamma\circ F_p\colon\C\to\C$. By
\eqref{eq:length}, $\varphi$ is a bijection which preserves the
Riemannian distances of $g_p$ and $g_q$. By the Myers--Steenrod
theorem~\cite{MS} (see also~\cite{Palais}), $\varphi$ is a smooth
diffeomorphism with $\varphi^*g_q=g_p$. By
Proposition~\ref{prop:conformal}, applied with $\lambda=1+|p'|^2$ and
$\nu=1+|q'|^2$, either $\varphi(z)=az+b$ or $\varphi(z)=a\bar z+b$, with
$a\ne 0$ and
\begin{equation}\label{eq:key}
  |a|^2\bigl(1+|q'(\varphi(z))|^2\bigr)=1+|p'(z)|^2 .
\end{equation}

\emph{Case $\varphi(z)=az+b$.} Put $f=p'$ and $g(z)=a\,q'(az+b)$. Then
\eqref{eq:key} reads $|f|^2-|g|^2=|a|^2-1$. Since $\deg p\ge 2$, $f$ is
not constant, and Lemma~\ref{lem:moduli} gives $|a|=1$ and $p'=v\,g$ with
$|v|=1$. Since $v\,g$ is the derivative of $v\,q(az+b)$, we get
$p(z)=v\,q(az+b)+k$ for some $k\in\C$. This is~(i) with $\alpha=v$,
$\beta=a$, $h=b$, and
$\Gamma(z,p(z))=F_q(\varphi(z))=(\beta z+h,\,q(\beta z+h))
=(\beta z+h,\,\bar\alpha(p(z)-k))$.

\emph{Case $\varphi(z)=a\bar z+b$.} Since $\conj{q'(w)}=\bar q'(\bar w)$,
we have $|q'(a\bar z+b)|=|\bar q'(\bar az+\bar b)|$. Put $f=p'$ and
$g(z)=a\,\bar q'(\bar az+\bar b)$. As before, Lemma~\ref{lem:moduli} gives
$|a|=1$ and $p'=v\,g$ with $|v|=1$. Since $v\,g$ is the derivative of
$(va/\bar a)\,\bar q(\bar az+\bar b)$, we get
$p(z)=\alpha\,\bar q(\beta z+h)+k$ with $\alpha=va/\bar a$, $\beta=\bar a$
and $h=\bar b$. Finally
$\Gamma(z,p(z))=(a\bar z+b,\,q(a\bar z+b))$, and
$q(a\bar z+b)=\conj{\bar q(\beta z+h)}=\conj{\bar\alpha(p(z)-k)}$, which
is~(ii).
\end{proof}

\begin{proof}[Proof of Theorem~\ref{thm:A}]
(c)$\Rightarrow$(b): the two maps displayed in Theorem~\ref{thm:A} are
isometries of $\C^2$, being a unitary diagonal map followed by a
translation, possibly composed with complex conjugation. If
$p(z)=\alpha\,q(\beta z+h)+k$, the first one sends $(z,p(z))$ to
$(\beta z+h,\,q(\beta z+h))$; if $p(z)=\alpha\,\bar q(\beta z+h)+k$, the
second one sends $(z,p(z))$ to $(u,q(u))$ with $u=\conj{\beta z+h}$. In
both cases the first coordinate runs over all of $\C$, so the map sends
$G_p$ onto $G_q$.

(b)$\Rightarrow$(a) is Lemma~\ref{lem:outer-inner}, and (a)$\Rightarrow$(c)
together with the final statement is Theorem~\ref{thm:inner-rigidity}.
\end{proof}

\begin{remark}[Invariants]\label{rem:invariants}
Theorem~\ref{thm:A} shows that $\deg p$ and the modulus $|p_n|$ of the
leading coefficient are isometry invariants of $G_p$. It also shows that
every isometry of $\C^2$ fixing the origin and mapping $G_p$ onto $G_q$
(with $p(0)=q(0)=0$) is a diagonal unitary map, possibly composed with
complex conjugation.
\end{remark}

\begin{remark}[Degree at most one]\label{rem:degree-one}
If $\deg p\le 1$, say $p'=a$, then $g_p=(1+|a|^2)\,|dz|^2$ and
$z\mapsto\sqrt{1+|a|^2}\,z$ is an isometry of $(\C,g_p)$ onto the
Euclidean plane; also $G_p$ is a real affine $2$-plane. Hence any two
graphs of degree at most $1$ are inner and outer isometric. By
Theorem~\ref{thm:inner-rigidity} applied to $\Gamma^{-1}$, a graph of
degree at most $1$ is not inner isometric to a graph of degree at least
$2$. This completes the classification.
\end{remark}

\begin{remark}[The real trace does not suffice]\label{rem:trace}
When $p$ has real coefficients, one may try to restrict an outer isometry
to the planar curve $\{(x,p(x)):x\in\R\}$. This information does not
determine the answer. For $p(x)=x^2$ and $q(y)=-2y^2+y$, the map
$E(x,y)=x\,(i/\sqrt2,\,i/\sqrt2)+y\,(0,1)$ is a linear isometric
embedding of $\R^2$ into $\C^2$ which sends $(x,x^2)$ to
$(ix/\sqrt2,\,x^2+ix/\sqrt2)=(u,q(u))$ with $u=ix/\sqrt2$; by
Lemma~\ref{lem:extension} it extends to an isometry of $\C^2$ mapping the
real trace of $G_p$ into $G_q$. However, $G_p$ and $G_q$ are not
isometric, since $|p_2|=1\ne 2=|q_2|$ (Remark~\ref{rem:invariants}).
\end{remark}

\section{The real case}\label{sec:real}

\subsection{Inner isometries}

For $p\in\R[x]$ let $\sigma_p(x)=\int_0^x\sqrt{1+p'(t)^2}\,dt$. Since
$\sigma_p'\ge 1$, $\sigma_p$ is an increasing bijection of $\R$ with
$|\sigma_p(x)|\ge|x|$.

\begin{lemma}\label{lem:arc}
The map $\Phi_p=\sigma_p\circ\pi\colon G_p\to\R$ is a bijection and
$\din(u,v)=|\Phi_p(u)-\Phi_p(v)|$ for all $u,v\in G_p$.
\end{lemma}

\begin{proof}
Let $u=F_p(x)$, $v=F_p(y)$ with $x\le y$. The curve
$t\mapsto F_p(x+t(y-x))$ has length $\sigma_p(y)-\sigma_p(x)$. Conversely,
if $\gamma=F_p\circ c$ joins $u$ to $v$, then by~\eqref{eq:length}
\[
  L(\gamma)=\int_0^1\sigma_p'(c(t))\,|c'(t)|\,dt
  =\int_0^1\bigl|(\sigma_p\circ c)'(t)\bigr|\,dt\ge\sigma_p(y)-\sigma_p(x).
  \qedhere
\]
\end{proof}

\begin{proof}[Proof of Theorem~\ref{thm:B}\textup{(a)}]
By Lemma~\ref{lem:arc}, $\Phi_q^{-1}\circ\Phi_p$ is an inner isometry of
$G_p$ onto $G_q$.
\end{proof}

\begin{remark}
For $p(x)=x^2$ and $q=0$ one gets the inner isometry
$(x,x^2)\mapsto(\sigma_p(x),0)$, with
$\sigma_p(x)=\frac12x\sqrt{1+4x^2}+\frac14\operatorname{arcsinh}(2x)$.
More generally, the inner isometries $G_p\to G_q$ are exactly the maps
$F_p(x)\mapsto F_q\bigl(\sigma_q^{-1}(\pm\sigma_p(x)+c)\bigr)$, $c\in\R$,
while, for $\deg p\ge 2$, there are at most two outer isometries
$G_p\to G_q$, by Corollary~\ref{cor:symmetry} below. The argument applies
verbatim to graphs of arbitrary $C^1$ functions $\R\to\R$.
\end{remark}

\subsection{Outer isometries}

\begin{lemma}\label{lem:degree-real}
Let $p,q\in\R[x]$ with $G_p$ and $G_q$ outer isometric. If $\deg p\ge 2$,
then $\deg q\ge 2$.
\end{lemma}

\begin{proof}
A graph $G_p$ is contained in a line if and only if $\deg p\le 1$: such a
line cannot be vertical, so it has the form $y=\alpha x+\beta$. Three
points of the plane are collinear if and only if one of their mutual
distances is the sum of the other two, a condition preserved by outer
isometries. If $\deg q\le 1$, any three points of $G_p$ would be
collinear, so $G_p$ would be contained in the line through two of its
points, which is impossible.
\end{proof}

\begin{proof}[Proof of Theorem~\ref{thm:B}\textup{(b)}]
If $p(x)=a\,q(bx+h)+k$, the isometry $(x,y)\mapsto(bx+h,\,a(y-k))$ sends
$(x,p(x))$ to $(bx+h,\,q(bx+h))$, and hence $G_p$ onto $G_q$.

Conversely, let $\Gamma\colon G_p\to G_q$ be an outer isometry. By
Lemma~\ref{lem:extension}, $\Gamma(v)=Av+c$ with $A\in O(2)$ and
$c=(c_1,c_2)$. Write $A=R_\theta\circ F^j$, where $R_\theta$ is the
rotation by $\theta$, $F(x,y)=(x,-y)$, $j\in\{0,1\}$, and put
$\varepsilon=(-1)^j$. Since $\Gamma(x,p(x))\in G_q$,
\begin{equation}\label{eq:real}
  q\bigl(s(x)\bigr)=\sin\theta\,x+\varepsilon\cos\theta\,p(x)+c_2,
  \qquad
  s(x)=\cos\theta\,x-\varepsilon\sin\theta\,p(x)+c_1 .
\end{equation}
Let $n=\deg p\ge 2$ and $m=\deg q$, so $m\ge 2$ by
Lemma~\ref{lem:degree-real}. If $\sin\theta\ne 0$, then $\deg s=n$ and
the left-hand side of~\eqref{eq:real} has degree $mn\ge 2n$, while the
right-hand side has degree at most $n$, a contradiction. Hence
$\sin\theta=0$, $b:=\cos\theta\in\{1,-1\}$, and \eqref{eq:real} becomes
$q(bx+c_1)=\varepsilon b\,p(x)+c_2$. With $a=\varepsilon b$ this gives
$p(x)=a\,q(bx+c_1)-a\,c_2$.
\end{proof}

The proof shows that every outer isometry $G_p\to G_q$, $\deg p\ge 2$, is
the restriction of a plane isometry $(x,y)\mapsto(bx+h,\,ay+k)$ with
$a,b\in\{1,-1\}$, and that $\deg p=\deg q$.

\begin{corollary}\label{cor:symmetry}
Let $p\in\R[x]$ with $n=\deg p\ge 2$, and let $\Sym(G_p)$ be the group of
isometries $\sigma$ of $\R^2$ with $\sigma(G_p)=G_p$. Then $\Sym(G_p)$ has
at most two elements. A non-trivial element, if any, is a reflection in a
vertical line when $n$ is even, and a half-turn about a point when $n$ is
odd.
\end{corollary}

\begin{proof}
Each $\sigma\in\Sym(G_p)$ has the form $\sigma(x,y)=(bx+h,\,ay+k)$, and
$\sigma(G_p)=G_p$ means $p(bx+h)=a\,p(x)+k$. Comparing leading
coefficients gives $a=b^n$. If $b=1$, then $a=1$ and $p(x+h)-p(x)=k$;
the coefficient of $x^{n-1}$ on the left is $n\,p_n\,h$, so $h=0$, then
$k=0$, and $\sigma=\mathrm{id}$. If $\sigma_1,\sigma_2$ both have $b=-1$,
then $\sigma_2^{-1}\sigma_1$ has $b=1$, so $\sigma_1=\sigma_2$. For $n$
even, $\sigma(x,y)=(h-x,\,y+k)$, and evaluating at $x=h/2$ gives $k=0$;
for $n$ odd, $\sigma(x,y)=(h-x,\,k-y)$ is the half-turn about
$(h/2,k/2)$.
\end{proof}

For instance, the graph of $x^2$ is symmetric with respect to the
$y$-axis, the graph of $x^3$ is symmetric under the half-turn about the
origin, and a direct computation shows that the graph of $x^4+x^3$ has no
non-trivial symmetries.

\begin{example}\label{ex:cubic}
Let $p(x)=x^3+x$ and $q(x)=x^3-x$. Since $i\,q(iz)=z^3+z$, the complex
graphs of $p$ and $q$ are isometric by Theorem~\ref{thm:A}. Their real
graphs are not outer isometric: in $a\,q(bx+h)+k$ the coefficient of
$x^2$ is $3a\,b^2h$, which forces $h=0$, and then the coefficient of $x$
is $-ab=-b^2=-1\ne 1$.
\end{example}

\section{Final remarks and questions}\label{sec:final}

Lemma~\ref{lem:extension} is classical; for isometries of a whole
Euclidean space it is a special case of the Mazur--Ulam
theorem~\cite{MU}, and we refer to~\cite{Berger} for Euclidean geometry
in general. We included its short proof because we need it for maps
defined only on a subset.

We end with some questions. Graphs of real polynomials in two variables
are surfaces in $\R^3$ with non-trivial inner geometry, but rigidity
fails in general: the graph of $p(x,y)=x^2$ is a parabolic cylinder, inner
isometric to the plane without being congruent to it. Which graphs of
real polynomials in two variables are isometrically rigid? What can be
said about graphs of polynomial maps $\C^n\to\C^m$, for the inner and the
outer distances? 

\section*{Declarations}

\subsection*{Disclosure of AI Usage.} The authors employed AI tools (Claude) for language editing, readability enhancements, and organization of ideas and proofs. The authors independently verified all mathematical results and take full responsibility for the accuracy and integrity of the content.

\subsection*{Data availability.} Data sharing is not applicable to this article as no datasets were generated or analyzed during the current study.

\subsection*{Acknowledgments.} The authors would like to thank Euripedes Silva and Joserlan da Silva for valuable discussions during the initial stage of this work.

\end{document}